\documentclass[12pt]{amsart}
\usepackage{amsfonts,amssymb,amscd,xy}
\usepackage[leqno]{amsmath}
\usepackage{mathrsfs}
\usepackage{amssymb}
\usepackage{amsmath}
\usepackage{amsxtra}
\usepackage{amsthm}
\usepackage{hyperref}

\newcommand{\Hom}{{\mathrm{Hom}}}

\newcommand\s{\mathscr }

\def\sp{S^{\e\prime}}
\def\spp{S^{\e\prime\prime}}
\def\sppp{S^{\e\prime\prime\prime}}

\def\g{\varGamma}

\def\aut{{\rm Aut}}

\def\xpp{X^{\le\prime\prime}}
\def\xppp{X^{\le\prime\prime\prime}}

\def\yp{Y^{\e\prime}}

\DeclareMathAlphabet{\mathbbmsl}{U}{bbm}{m}{sl}

\newcommand{\isoto}{\overset{\!\sim}{\to}}

\usepackage{tabularx}
\usepackage{epsfig,amssymb}
\usepackage{bm} 
\usepackage{verbatim} 

\usepackage{hyperref}
\usepackage[color,notcite,notref]{showkeys}

\usepackage{verbatim} 

\makeatletter

\DeclareMathAlphabet{\mathcalligra}{T1}{calligra}{m}{n}

\numberwithin{equation}{section}

\def\be{\kern -.1em}
\def\le{\kern 0.03em}
\def\lle{\kern 0.015em}
\def\lbe{\kern -.025em}

\newcommand{\spec}{\mathrm{ Spec}\,}

\def\e{\kern 0.08em}

\newtheorem{lemma}{Lemma}[section]

\newtheorem{proposition-definition}[lemma]{Proposition-Definition}
\newtheorem{corollary}[lemma]{Corollary}
\newtheorem{proposition}[lemma]{Proposition}
\theoremstyle{definition}

\theoremstyle{remark}
\newtheorem{remark}[lemma]{Remark}

\newtheorem{example}[lemma]{Example}

\begin{document}
	
\input xy     
\xyoption{all}

\title[Remarks on descent data]{Some remarks on descent data with applications to Galois descent}

\author{Cristian D. Gonz\'alez-Avil\'es}
\address{Departamento de Matem\'aticas, Universidad de La Serena, Cisternas 1200, La Serena 1700000, Chile}
\email{cgonzalez@userena.cl}
\thanks{Partially supported by Fondecyt grant 1160004.}
\date{\today}


\maketitle
	
\begin{abstract} We present a proof of the equivalence of the standard definition of descent data on schemes with another one mentioned in the literature that involves certain cartesian diagrams. Using this equivalence, we discuss the Galois descent of both schemes and morphisms of schemes.
\end{abstract}

\topmargin -1cm

\smallskip

\maketitle

\topmargin -1cm

\setcounter{section}{-1}

\section{Introduction}

In this expository paper we present a detailed proof of the equivalence of the standard definition of descent data on schemes with another one mentioned in the literature that involves certain cartesian diagrams. See Section \ref{two}. To our knowledge, no detailed proof of this equivalence has appeared in print. As an application, we provide in Section \ref{3} the missing details of the discussion of Galois descent contained in \cite[\S6.2, Example B, pp.~139-140]{blr}. The subject of Galois descent is discussed amply in the literature, but mostly over a field. Even when a more general base scheme is allowed, some important details are omitted (for another example of such omissions, see \cite[comments after (14.20.1), p.~457]{gw}). In Section \ref{4}, as an application of the detailed discussion of Section \ref{3}, we generalize the standard result on Galois descent for morphisms of schemes over a field \cite[Proposition 2.8]{j} to an arbitrary base scheme. More precisely, we show that, if $\sp\to S$ is a finite Galois covering of schemes with Galois group $\g$ and $\delta\colon X_{1}\to X_{2}$ is an $\sp$-morphism of $\g$-schemes that descend to $S$, then $\delta$ descends to $S$ if, and only if, $\delta$ is invariant under the action of $\g$ on morphisms defined by \eqref{act3}.

\section*{Acknowledgement}

I thank Mikhail Borovoi for his constructive criticism of the first version of this paper.

\section{Preliminaries}\label{pre}

The identity morphism of an object $A$ of a category will be denoted by $1_{\be A}$.

If $S$ is a scheme, $({\rm Sch}/S\le)$ will denote the category of $S$-schemes. If $X$ is an $S$-scheme, $\aut(\le X\!/\lbe S\e)$ will denote the group of $S$-automorphisms of $X$.

Given morphisms of schemes $X\overset{f}{\to}S\overset{\,\,g}{\leftarrow}T$, we will write $X\be\times_{f,\le S,\e g}\be T$ for the fiber product of $f$ and $g$. When $f$ and $g$ are not relevant, we will write  $X\times_{S}T$ for $X\be\times_{f,\e S,\le g}\be T$.
If $u\colon X\to Y$ is an $S$-morphism of schemes, $u\times_{S}T$ will denote the $T$-morphism of schemes $u\times_{S}1_{\lbe T}\colon X\times_{S}T\to Y\times_{S}T$. 

\smallskip

Recall that a commutative diagram in a category $\s C$
\begin{equation}\label{squ}
\xymatrix@1@R=30pt@C=40pt{
Z\ar[r]^{u} \ar[d]_{v}& U \ar[d]^{\pi}\\
V\ar[r]^{w} & W
}
\end{equation}
is {\it cartesian} if for every commutative diagram of solid arrows in $\s C$
\begin{equation}\label{full}
\xymatrix{T\ar@{-->}[dr]_(.4){h}\ar@/^.9pc/[drrr]\ar@/_1.2pc/[ddr]&&\\
& Z\ar[d]\ar[rr]&& U\ar[d]\\
& V\ar[rr]&& W}
\end{equation}
there exists a unique arrow $h\colon T\to Z$ in $\s C$ such that the full diagram \eqref{full} commutes. It is easy to check that \eqref{squ} is cartesian if both horizontal arrows $u$ and $w$ are isomorphisms. Further, if \eqref{squ} is cartesian and $\psi\colon Y\isoto Z$ is an isomorphism in $\s C$, then 
\begin{equation}\label{squ2}
\xymatrix@1@R=30pt@C=40pt{
Y\ar[r]^{u\e\circ\e\psi} \ar[d]_{v\e\circ\e\psi}& U \ar[d]^{\pi}\\
V\ar[r]^{w} & W
}
\end{equation}
is cartesian as well. We will also need the following fact.
\begin{lemma}\label{ucart} If
\[
\xymatrix@1@R=30pt@C=40pt{
Z_{i}\ar[r]^{g_{i}} \ar[d]_{e_{i}}& U \ar[d]\\
V_{i}\ar[r]^{k_{i}} & W
}
\]
is a cartesian square in the category of schemes for every $i$ in some index set, then the diagram
\[
\xymatrix@1@R=30pt@C=40pt{
\coprod Z_{i}\ar[r]^{\coprod g_{i}} \ar[d]_(.5){\coprod e_{i}}& U \ar[d]\\
\coprod V_{i}\ar[r]^{\coprod k_{i}} & W.
}
\]
is cartesian as well.
\end{lemma}
\begin{proof} (Sketch)  This may be verified by starting with a commutative diagram
\[
\xymatrix{T\ar@{-->}[dr]\ar@/^1pc/[drrr]^{t_{1}}\ar@/_1.5pc/[ddr]_{t_{2}}&&\\
& \coprod Z_{i}\ar[d]_{\coprod e_{i}}\ar[rr]^{\coprod g_{i}}&& U\ar[d]\\
& \coprod V_{i}\ar[rr]^{\coprod k_{i}}&& W}
\]
and considering, for every $j$, the commutative diagram
\[
\xymatrix{T\times_{\coprod\be V_{i}}\be V_{j}\ar[dr]_(.56){\exists!\, h_{j}}\ar@/^1pc/[drrr]^{t_{1}\e\circ\e{\rm pr}_{1}}\ar@/_1.5pc/[ddr]_(.45){t_{2}\times_{\coprod\be V_{i}}1_{\lle V_{\lbe j}}}&&\\
& Z_{j}\ar[d]_{ e_{j}}\ar[rr]^{g_{j}}&& U\ar[d]\\
& V_{j}\ar[rr]^{k_{j}}&& W.}
\]
\end{proof}

For more information on fiber products and coproducts of schemes and cartesian diagrams, see  \cite[Chapter 0, \S1.2, and Chapter I, \S3.1]{ega1}.

\smallskip

\section{Descent data on schemes}\label{two}

In this Section we reformulate the standard definitions of {\it covering data} and {\it descent data} on schemes (these standard definitions can be found, for example, in \cite[Chapter 6]{blr}). We focus on schemes, but similar considerations apply to quasi-coherent modules.

\medskip

Let $f\colon\sp\to S$ be a morphism of schemes, set $\spp=\sp\times_{S}\sp$ and let $p_{\le i}\colon \spp\to\sp$ ($i=1,2$) be the canonical projection onto the $i$-th factor. Then the following diagram is cartesian 
\begin{equation}\label{can}
\xymatrix@1@R=30pt@C=40pt{
\spp\ar[r]^(.5){p_{\le 1}} \ar[d]_{p_{2}}& \sp \ar[d]^{f}\\
\sp\ar[r]^(.5){f} & S.
}
\end{equation}
Further, set $\sppp=\sp\times_{S}\e\sp\times_{S}\e\sp$ and let $p_{ jk}\colon\sppp\to\spp$ be given (set-theoretically) by $(s_{1},s_{2},s_{3})\mapsto (s_{j},s_{k})$, where
$(\e j,k)=(1,2), (1,3)$ or $(2,3)$. Note that
\begin{eqnarray}\label{eqs1}
p_{\le 1}\circ p_{\le 12}&=&p_{\le 1}\circ p_{\le 13}\\
p_{\le 1}\circ p_{\le 23}&=&p_{\le 2}\circ p_{\le 12}\label{eqs2}\\
p_{\le 2}\circ p_{\le 23}&=&p_{\le 2}\circ p_{\le 13},\label{eqs3}
\end{eqnarray}
where the first (respectively, second, third) common composition $\sppp\to\sp$ is the projection onto the first (respectively, second, third) factor.

If $\pi\colon X\to\sp$ is an $\sp$-scheme and $i=1$ or 2, we will write $p_{\le i}^{\lle *}X=X\times_{S^{\lle\prime}\be,\e p_{\lle i}}\!\spp$ and regard it as an $\spp$-scheme via $p_{\le i}^{\lle *}(\pi)$. Further, we will write $p_{\e i,\le X}=(\, p_{\le i})_{\lbe X}\colon p_{\le i}^{*}X\to X$. The $\sppp$-schemes $p_{\lbe jk}^{\le *}\e p_{\le i}^{\lle *}X$ (with structural morphisms $p_{\lbe jk}^{\le *}\e p_{\le i}^{\lle *}(\pi)$) and morphisms $p_{\le jk,\e p_{i}^{*}\lbe X}\colon
p_{\lbe jk}^{\le *}\e p_{\le i}^{\lle *}X\to p_{\le i}^{\lle *}X$ are defined similarly. The equalities \eqref{eqs1}-\eqref{eqs3} induce various identifications among these objects. For example, by \eqref{eqs3},
\begin{equation}\label{id0}
p_{\le 23,\e p_{2}^{*}X}=p_{\le 13,\e p_{2}^{*}X}.
\end{equation}

Recall that a {\it covering datum on $X$ relative to $f$} is an isomorphism of $\spp$-\e schemes $p_{1}^{*}X\simeq p_{\le 2}^{*}\le X$.

Set $\xpp=p_{1}^{*}X$ and let $\varphi\colon \xpp\isoto p_{2}^{*}X$ be a covering datum on $X$. In particular, $p_{\le 1}^{\lle *}(\pi)=p_{\le 2}^{\lle *}(\pi)\circ\varphi$. Define 
\begin{eqnarray}\label{q1}
q_{\le 1}&=&p_{\le 1,\le X}\\
\label{q2}
q_{\le 2}&=&p_{\le 2,\le X}\circ\varphi.
\end{eqnarray}
Then the following diagram, which is an instance of diagram \eqref{squ2}, is cartesian for $i=1$ and $2$:
\begin{equation}\label{cart}
\xymatrix@1@R=30pt@C=40pt{
\xpp\ar[r]^(.5){q_{\le i}} \ar[d]_(.45){p_{1}^{\lle *}(\pi)}& X \ar[d]^(.45){\pi}\\
\spp\ar[r]^(.5){p_{\le i}} &\sp.
}
\end{equation}

\smallskip

Conversely, assume that there exist cartesian diagrams of the form \eqref{cart} such that \eqref{q1} holds. Then there exist unique morphisms $\varphi\colon \xpp\to p_{2}^{*}\le X$ and $\psi\colon p_{2}^{*}\le X\to\xpp$ such that the following diagrams commute:
\[
\xymatrix{\xpp\ar[dr]_(.5){\varphi}\ar@/^1pc/[drrr]^(.5){q_{\lle 2}}\ar@/_1.4pc/[ddr]_(.4){p_{1}^{\lle *}\be(\pi)}&&\\
&p_{2}^{*}X\ar[d]_(.45){p_{2}^{\lle *}\lbe(\pi)}\ar[rr]^{p_{\le 2,\le X}}&& X\ar[d]^(.45){\pi}\\
&\spp\ar[rr]^(.48){p_{\lle 2}}&&
\sp}
\]
and
\[
\xymatrix{p_{2}^{*}\le X\ar[dr]_(.5){\psi}\ar@/^.9pc/[drrr]^{p_{\le 2,\lle X}}\ar@/_1.4pc/[ddr]_(.45){p_{2}^{\lle *}\be(\pi)}&&\\
&\xpp\ar[d]_(.45){p_{1}^{\lle *}\be(\pi)}\ar[rr]^(.48){q_{\lle 2}}&& X\ar[d]^(.45){\pi}\\
&\spp\ar[rr]^(.48){p_{\lle 2}}&&
\sp.}
\]
Since the diagrams
\[
\xymatrix{p_{2}^{*}\le X\ar[dr]_(.5){\varphi\e\circ\e\psi}\ar@/^.9pc/[drrr]^{p_{\le 2,\lle X}}\ar@/_1.4pc/[ddr]_(.45){p_{2}^{\lle *}\be(\pi)}&&\\
&p_{2}^{*}X\ar[d]_(.45){p_{2}^{\lle *}\lbe(\pi)}\ar[rr]^{p_{\le 2,\lle X}}&& X\ar[d]^(.45){\pi}\\
&\spp\ar[rr]^(.48){p_{\lle 2}}&&
\sp}
\]
and 
\[
\xymatrix{\xpp\ar[dr]_(.5){\psi\e\circ\e\varphi}\ar@/^.9pc/[drrr]^(.5){q_{\lle 2}}\ar@/_1.4pc/[ddr]_{p_{1}^{\lle *}\be(\pi)}&&\\
&\xpp\ar[d]_(.4){p_{1}^{\lle *}\lbe(\pi)}\ar[rr]^{q_{\lle 2}}&& X\ar[d]^(.45){\pi}\\
&\spp\ar[rr]^(.48){p_{\lle 2}}&&
\sp}
\]
commute, $\varphi\le\circ\le \psi$ (respectively, $\psi\circ \varphi$) is the identity morphism of $p_{2}^{*}X$ (respectively, $\xpp\e$). Thus we obtain an $\spp$-isomorphism $\varphi\colon \xpp\isoto p_{\le 2}^{*}X$ (i.e., a covering datum on $X$ relative to $f\e$) such that \eqref{q2} holds.

\smallskip

We conclude that to give a covering datum on $X$ relative to $f$ is equivalent to giving a pair of cartesian diagrams \eqref{cart} such that \eqref{q1} holds.

\smallskip

Now let $\varphi\colon \xpp\isoto p_{2}^{*}X$ again be a covering datum on $X$ relative to $f$ and define $q_{\le 1}$ and $q_{\le 2}$ by \eqref{q1} and \eqref{q2}, respectively. Further, write $\xppp=p_{\le 12}^{*}\e p_{\le 1}^{*}X=p_{\le 13}^{*}\e p_{\le 1}^{*}\le X$ \eqref{eqs1} and note that $p_{\le 12}^{*}\e p_{\le 2}^{*}X=\le p_{\le 23}^{*}\le\xpp$ \eqref{eqs2}. Then we may discuss the $\sppp$-isomorphism $p_{\le 12}^{*}\e\varphi\colon \xppp\isoto p_{\le 23}^{*}\le\xpp$ in analogy to the foregoing discussion of the $\spp$-isomorphism $\varphi\colon \xpp\isoto p_{2}^{*}\le X$. Thus we define
\begin{eqnarray}\label{q12}
q_{12}&=&p_{\le 12,\le X^{\prime\prime}}\\
\label{q13}
q_{13}&=& p_{\le 13,\le X^{\prime\prime}}\\
\label{q23}
q_{\le 23}&=& p_{\le 23,\le X^{\prime\prime}}\le\circ\le p_{\le 12}^{*}\e\varphi.
\end{eqnarray}
Since $p_{12}^{*}\e p_{1}^{*}\lbe(\pi)=p_{13}^{*}\e p_{1}^{*}\lbe(\pi)=p_{\le 12}^{*}\e p_{\le 2}^{*} \lbe(\pi)\circ\e p_{\le 12}^{*}\varphi=p_{\le 23}^{*}\e p_{\le 1}^{*} \lbe(\pi)\circ\e p_{\le 12}^{*}\varphi$, the following diagram is cartesian for $(\e j,k)=(1,2), (1,3)$ and $(2,3)$ as an instance of diagram \eqref{squ2}:  
\[
\xymatrix@1@R=30pt@C=40pt{
\xppp\ar[r]^(.5){q_{\le jk}} \ar[d]_{p_{12}^{*}\e p_{1}^{*}\lbe(\pi)}& \xpp\ar[d]^{p_{1}^{*}(\pi)}\\
\sppp\ar[r]^(.5){p_{jk}} &\spp.
}
\]
Now, by the commutativity of 
\begin{equation}\label{pain}
\xymatrix@1@R=30pt@C=40pt{
p_{jk}^{*}\xpp\ar[rr]^(.5){p_{jk,\le X^{\prime\prime}}} \ar[d]_{p_{ jk}^{*}\varphi}^{\sim}&& \xpp\ar[d]^{\varphi}_{\sim}\\
p_{jk}^{*}p_{2}^{*}X\ar[rr]^(.5){p_{jk,\le p_{2}^{*}X}} && p_{2}^{*}X
}
\end{equation}
and the equalities \eqref{q1}, \eqref{q2}, \eqref{q12}, \eqref{q13} and \eqref{q23}, we have
\begin{eqnarray}\label{eq1}
q_{\le 1}\circ q_{\le 12}&=&q_{\le 1}\circ q_{\le 13}\\
q_{\le 1}\circ q_{\le 23}&=&q_{\le 2}\circ q_{\le 12}.\label{eq2}
\end{eqnarray}
Assume now that $\varphi$ is, in fact, a {\it descent datum on $X$ relative to $f$}, i.e., the following diagram of isomorphisms of $\sppp$-schemes, where the equalities are induced by \eqref{eqs1}, \eqref{eqs2} and \eqref{eqs3}, commutes:
\begin{equation}\label{coc}
\xymatrix{
p_{\le 12}^{*}\e p_{1}^{*}X=p_{\le 13}^{*}\e p_{1}^{*}\le X	\ar[rr]^(0.5){p_{\le 13}^{*}\e\varphi}_{\sim}\ar[dr]_(.4){p_{\le 12}^{*}\e\varphi}^(.5){\sim}  && p_{\le 23}^{*}\e p_{2}^{*}\le X=p_{\le 13}^{*}\e p_{2}^{*}\le X\\
& p_{\le 12}^{*}\e p_{2}^{*}\le X=p_{\le 23}^{*}\e p_{1}^{*}X\ar[ur]_{p_{\le 23}^{*}\e\varphi}^(.4){\sim}\,. &
}
\end{equation}
Then 
\begin{equation}\label{eq3}
q_{\le 2}\circ q_{\le 23}=q_{\le 2}\circ q_{\le 13}.
\end{equation} 
Indeed, by \eqref{id0}, \eqref{q2}, \eqref{q13}, \eqref{q23}, \eqref{pain} and \eqref{coc}, we have
\[
\begin{array}{rcl}
q_{\le 2}\circ q_{\le 23}&=&p_{\e 2,\le X}\le\circ\le(\varphi\circ p_{\le 23,\le X^{\prime\prime}})\le\circ\le p_{\le 12}^{*}\e\varphi=
p_{\e 2,\le X}\circ (\e p_{\le 23,\le p_{2}^{*}X}\circ p_{\le 23}^{*}\e\varphi)\circ p_{\le 12}^{*}\e\varphi\\
&=& p_{\e 2,\le X}\circ (\e p_{\le 13,\e p_{2}^{*}X}\circ p_{\le 13}^{*}\e\varphi)=p_{\e 2,\le X}\circ\varphi\circ p_{\le 13,\le X^{\prime\prime}}=q_{\le 2}\circ q_{\le 13}.
\end{array}
\]
Thus we obtain six commutative diagrams
\begin{equation}\label{mcart}
\xymatrix@1@R=30pt@C=40pt{
\xppp\ar[r]^(.5){q_{\le jk}} \ar[d]_{p_{12}^{*}\e p_{1}^{*}\lbe(\pi)}& \xpp\ar[r]^(.5){q_{\le i}}\ar[d]_{p_{1}^{*}\lbe(\pi)}& X\ar[d]^{\pi}\\
\sppp\ar[r]^(.5){p_{jk}} &\spp\ar[r]^(.5){p_{\le i}}&\sp\,,\\
}
\end{equation}
where $i=1$ or $2$, $(\e j,k)=(1,2), (1,3)$ or $(2,3)$, the squares are cartesian, equations \eqref{q1}, \eqref{q2}, \eqref{q12}, \eqref{q13} and \eqref{q23} hold (where $\varphi$ is the covering datum on $X$ determined by the right-hand square in \eqref{mcart} for $i=2$) and the various top horizontal compositions satisfy the relations \eqref{eq1}, \eqref{eq2} and \eqref{eq3}.

\smallskip

Conversely, assume that there exist commutative diagrams of the form \eqref{mcart} with cartesian squares such that 
\eqref{q1}, \eqref{q2} (where $\varphi$ is the covering datum on $X$ determined by the right-hand square in \eqref{mcart} for $i=2$), \eqref{q12}, \eqref{q13}, \eqref{eq2} and \eqref{eq3} hold. Then \eqref{eq1} also holds since it follows from \eqref{eqs1}, \eqref{q1}, \eqref{q12} and \eqref{q13}. We will show that \eqref{q23} holds as well and that diagram \eqref{coc} commutes, i.e., $\varphi$ is a descent datum on $X$ relative to $f$.

By \eqref{q2}, \eqref{q12} and the commutativity of \eqref{pain}, the following diagram commutes
\[
\xymatrix{\xppp\ar[dr]_(.5){p_{\le 12}^{*}\varphi}\ar@/^.9pc/[drrr]^(.6){\varphi\e\circ\e q_{\lle 12}}\ar@/_2.1pc/[ddr]_(.4){p_{12}^{*}\e p_{1}^{*}\lbe(\pi)}\ar@/^1.8pc/[drrrrr]^(.6){q_{\le 2}\e\circ\e q_{\le 12}}&&&\\
&p_{\le 12}^{*}\e p_{2}^{*}\le X\ar[d]_{p_{12}^{*}\e p_{2}^{*}\lbe(\pi)}\ar[rr]^{p_{\le 12,\e p_{2}^{*}X}}&& p_{2}^{*}\le X\ar[rr]^(.48){p_{\le 2, X}}\ar[d]_{p_{2}^{*}(\pi)}&& X\ar[d]^{\pi}\\
&\sppp\ar[rr]^(.48){p_{\lle 12}}&&
\spp\ar[rr]^(.48){p_{\lle 2}}&&\sp.}
\]
Further, $p_{\le 12}^{*}\varphi$ is the unique morphism such that $p_{\le 2, X}\circ
p_{\le 12,\e p_{2}^{*}X}\circ p_{\le 12}^{*}\varphi=q_{\le 2}\e\circ\e q_{\le 12}$. Similarly, there exists an $\sppp$-isomorphism $g\colon\xppp\isoto p_{\le 23}^{*}\xpp=p_{\le 12}^{*}\e p_{2}^{*}\le X$ such that the following diagram commutes
\[
\xymatrix{\xppp\ar[dr]_(.5){g}\ar@/^.9pc/[drrr]^(.6){q_{\lle 23}}\ar@/_2.1pc/[ddr]_(.4){p_{12}^{*}\e p_{1}^{*}\lbe(\pi)}\ar@/^1.8pc/[drrrrr]^(.6){q_{\le 1}\e\circ\e q_{\le 23}}&&&\\
&p_{\le 23}^{*}\xpp\ar[d]_{p_{23}^{*}\e p_{1}^{*}\lbe(\pi)}\ar[rr]^{p_{\le 23,\le X^{\le\prime\prime}}}&& \xpp\ar[rr]^(.48){p_{\le 1, X}}\ar[d]_{p_{1}^{*}(\pi)}&& X\ar[d]^{\pi}\\
&\sppp\ar[rr]^(.48){p_{\lle 23}}&&
\spp\ar[rr]^(.48){p_{\lle 1}}&&\sp,}
\]
where we have used \eqref{q1}. Moreover, $g$ is the unique morphism that satisfies the identity $p_{\le 1, X}\circ
p_{\le 23,X^{\le\prime\prime}}\circ g=q_{\le 1}\e\circ\e q_{\le 23}$. Now  \eqref{eqs2}, \eqref{eq2} and the preceding uniqueness statements imply that $g=p_{\le 12}^{*}\varphi$, whence $q_{\le 23}=p_{\le 23,\le X^{\prime\prime}}\le\circ\le p_{\le 12}^{*}\e\varphi$, i.e., \eqref{q23} holds. Finally, the diagram with cartesian square (where the equalities come from \eqref{eqs3} and \eqref{eq3})
\[
\xymatrix{\xppp\ar[dr]^(.45){h}\ar@/^1.6pc/[drrrrr]^(.6){\hskip .7cm q_{\le 2}\e\circ\e q_{\le 23}=q_{\le 2}\e\circ\e q_{\le 13}}\ar@/_1.8pc/[ddr]_{p_{12}^{*}\e p_{1}^{*}\lbe(\pi)}&&&&\\
&p_{\le 23}^{*}\e p_{\le 2}^{*}X=p_{\le 13}^{*}\e p_{\le 2}^{*}X\ar[d]^{p_{23}^{*}\e p_{2}^{*}\lbe(\pi)=p_{13}^{*}\e p_{2}^{*}\lbe(\pi)}\ar[rrrr]^(.55){(\e p_{\le 2}\e\circ\e p_{\le 23})_{X}=(\e p_{\le 2}\e\circ\e p_{\le 13})_{X}}&&&& X\ar[d]^{\pi}\\
&\sppp\ar[rrrr]_(.48){p_{\le 2}\e\circ\e p_{\le 23}=p_{\le 2}\e\circ\e p_{\le 13}}&&&&
\sp}
\]
commutes for $h=p_{\le 23}^{*}\varphi\circ p_{\le 12}^{*}\varphi$ and $h=p_{\le 13}^{*}\varphi$. Indeed, since \eqref{pain} commutes and \eqref{q2}, \eqref{q13} and \eqref{q23} hold, we have
\[
(\e p_{\le 2}\e\circ\e p_{\le 23})_{X}\circ p_{\le 23}^{*}\varphi\circ p_{\le 12}^{*}\varphi=p_{\le 2,\e X}\circ\varphi\circ p_{\le 23,\e X^{\prime\prime}}
\circ p_{\le 12}^{*}\varphi=q_{\le 2}\circ q_{\le 23}
\]
and
\[
(\e p_{\le 2}\e\circ\e p_{\le 13})_{X}\circ p_{\le 13}^{*}\varphi=p_{\le 2,\e X}\circ\varphi\circ p_{\le 13,\e X^{\prime\prime}}
=q_{\le 2}\circ q_{\le 13}.
\]
Thus $p_{\le 23}^{*}\varphi\circ p_{\le 12}^{*}\varphi=p_{\le 13}^{*}\varphi$, i.e., the cocycle condition \eqref{coc} is satisfied.

\medskip

We conclude that to give a descent datum on $X$ relative to $f$ is equivalent to giving six commutative diagrams of the form \eqref{mcart} consisting of cartesian squares such that \eqref{q1}, \eqref{q2} (where $\varphi$ is the covering datum on $X$ determined by the right-hand square in \eqref{mcart} for $i=2$), \eqref{q12}, \eqref{q13}, \eqref{eq2} and \eqref{eq3} hold.

\medskip

To conclude this Section, we observe that, if $Y$ is an $S$-scheme, then the $\sp$-scheme $\yp=f^{*}Y=Y\times_{S}\sp$ is endowed with a canonical descent datum $c_{\e Y}\colon
p_{\le 1}^{*}\yp\isoto p_{\le 2}^{*}\yp$, namely the composite $\spp$-isomorphism
\[
p_{\le 1}^{*}\yp=p_{\le 1}^{*}f^{*}Y\simeq (\e f\circ p_{\le 1})^{*}(Y)=
(\e f\circ p_{\le 2})^{*}(Y)\simeq p_{\le 2}^{*}\yp,
\]
where the second equality holds by the commutativity of \eqref{can}. Set-theoretically, $c_{\e Y}$ can be described by the formula
\begin{equation}\label{cy}
c_{\le Y}\be(\le y,s^{\lle\prime},s^{\lle\prime},t^{\le\prime}\le)=(y,t^{\lle\prime},s^{\lle\prime},t^{\le\prime}\le),
\end{equation}
where $(y,-)\in\yp$ and $(s^{\lle\prime},t^{\le\prime}\le)\in\spp$.

\section{Galois descent of schemes}\label{3}

In this Section we use the developments of the previous Section to discuss Galois descent of schemes. Compare with \cite[\S6.2, Example B, pp.~139-140]{blr}. 

\smallskip

Recall that a morphism of schemes $f\colon\sp\to S$ is said to be {\it finite and locally free} if $f$ is affine and $f_{*}\mathcal O_{S^{\lle\prime}}$ is a finite and locally free $\mathcal O_{S}$-module. Equivalently, $f$ is finite, flat and locally of finite presentation.

\smallskip

Let $f\colon\sp\to S$ be a finite, surjective and locally free morphism (in particular, $f$ is faithfully flat and quasi-compact) and let $\g$ be a subgroup of $\aut\!\left(\sp\!/S\le\right)$. If $X$ is an $S$-scheme, an {\it action of $\g$ on $X$ over $S$} (via automorphisms) is a group homomorphism $\rho\colon\g\to\aut(X\be/S\e)$.

For every scheme $X$, set $\g\times X=\coprod_{\,\sigma\in\g}X$. Then
$\rho$ induces an action of the $S$-group scheme $\g\times S$ on $X$ over $S$, i.e., an $S$-morphism $(\g\times S)\times_{S}X\to X$ subject to well-known conditions. We will henceforth identify $(\g\times S)\times_{S}X$ and $\g\times X$ so that the preceding morphism will be written as $\g\times X\to X$.
Now set
\[
b=\coprod_{\,\sigma\in\g}1_{\le S^{\lle\prime}}\colon \g\times \sp\to \sp, (\sigma,s^{\le\prime}\e)\mapsto s^{\le\prime}.
\]
We will regard $\g\times\sp$ as an $S$-scheme via $f\circ b$ (whence $b$ is an $S$-morphism). The canonical action of $\g$ on $\sp$ over $S$, i.e., the $S$-morphism $\coprod_{\,\sigma\in\g}\sigma\colon\coprod_{\,\sigma\in\g}\sp\to\sp $ will be written as
\begin{equation}\label{a}
a\colon \g\times \sp\to \sp, (\sigma,s^{\le\prime}\e)\mapsto \sigma\le s^{\le\prime}.
\end{equation}

We now assume that $f$ is a {\it Galois covering} with Galois group $\g$, i.e., the morphism of $S$-schemes
\begin{equation}\label{gal}
\vartheta=(b,a)_{S}\colon \g\times\sp\to\spp, \,(\sigma,s^{\le\prime}\e)\mapsto (s^{\le\prime}, \sigma\le s^{\le\prime}\e),
\end{equation}
is an {\it isomorphism}.

For example, if $K/k$ is a finite Galois extension of fields with Galois group $\g$, then the canonical morphism $f\colon\spec K\to\spec k$ is a Galois covering. In effect, in this case $\vartheta$ \eqref{gal} is the isomorphism of $k$-schemes induced by the isomorphism of $k$-algebras $K\otimes_{\e k}K\isoto \prod_{\e\sigma\in\g} K, x\otimes y\mapsto \prod_{\e\sigma\in\g} x\le\sigma(y)$.

Clearly, the following diagrams commute
\begin{equation}\label{gal3}
\xymatrix{
\g\times\sp\ar[rr]^(0.5){\vartheta}_{\!\sim}\ar[dr]_(.4){b}  && \spp\ar[dl]^(.4){p_{\lle 1}} \\
& \sp &
}
\end{equation}
and
\begin{equation}\label{gal2}
\xymatrix{
\g\times\sp\ar[rr]^(0.5){\vartheta}_{\!\sim}\ar[dr]_(.4){a}  && \spp\ar[dl]^(.4){p_{\lle 2}} \\
& \sp\,. &
}
\end{equation}

Further, since \eqref{gal} is an isomorphism, the morphism of $S$-schemes
\begin{equation}\label{galt}
\varrho\colon \g\times\g\times\sp\to\sppp, \,(\sigma,\tau,s^{\le\prime}\e)\mapsto (s^{\le\prime}, \tau s^{\le\prime},(\sigma\tau) s^{\le\prime}\e),
\end{equation}
is an isomorphism as well.

We now define $S$-morphisms $\widetilde{p}_{jk}\colon\g\times\g\times\sp\to\g\times\sp$ by the formulas
\begin{eqnarray}\label{pjk1}
\widetilde{p}_{12}(\sigma,\tau,s^{\le\prime})&=&(\tau,s^{\le\prime}\e),\\
\label{pjk2}
\widetilde{p}_{13}(\sigma,\tau,s^{\le\prime})&=&(\sigma\tau,s^{\le\prime}\e),\\
\label{pjk3}
\widetilde{p}_{23}(\sigma,\tau,s^{\le\prime})&=&(\sigma,\tau s^{\le\prime}\e).
\end{eqnarray}
Then the following diagram commutes for $(\e j,k)=(1,2), (1,3)$ and $(2,3)$:
\begin{equation}\label{comm}
\xymatrix@1@R=30pt@C=40pt{
\g\times\g\times\sp\ar[rr]^(.5){\varrho}_{\sim} \ar[d]_{\widetilde{p}_{jk}}&& \sppp \ar[d]^{p_{jk}}\\
\g\times\sp\ar[rr]^(.5){\vartheta}_{\sim} &&\spp
}
\end{equation}

\smallskip

Let $\pi\colon X\to\sp$ be an $\sp$-scheme and recall the schemes $\xpp=p_{\le 1}^{\lle *}\le X=X\times_{\pi,\e S^{\lle\prime}\lbe,\e p_{1}}\spp$ and $\xppp=p_{\le 12}^{*}\e p_{1}^{*}\lbe X=p_{\le 13}^{*}\e p_{1}^{*}\lbe X$. We will make the identifications
\begin{eqnarray*}
\g\times X&=& X\times_{S^{\lle\prime}\lbe,\e b}(\g\times \sp)=b^{*}\be X\\
1_{\g}\times\pi&=&b^{*}\lbe(\pi)=\vartheta^{*}\be(\e p_{\le 1}^{*}(\pi))\quad(\text{see \eqref{gal3}})\\
\g\times\g\times X&=&(\e b\circ \widetilde{p}_{12}\e)^{*}\lbe X=(\e b\circ \widetilde{p}_{13}\e)^{*}\lbe X\\
1_{\g}\times 1_{\g}\times\pi&=&(\e b\circ \widetilde{p}_{12}\e)^{*}(\pi)=\widetilde{p}_{12}^{\,\e *}(\e \vartheta^{\e *}\lbe p_{\le 1}^{*}(\pi))=\varrho^{*}\be(\e p_{\le 12}^{*}\e p_{\le 1}^{*}(\pi))\quad(\text{see \eqref{comm}}).
\end{eqnarray*}
Via the above identifications, $\vartheta$ and $\varrho$ induce isomorphisms
\begin{equation}\label{galp}
\vartheta_{\lbe X^{\prime\prime}}\colon \g\times X\isoto \xpp,\, (\sigma,x)\mapsto (x,\pi(x),\sigma\pi(x)),
\end{equation}
and
\begin{equation}\label{galp2}
\varrho_{\lle X^{\prime\prime\prime}}\colon \g\times\g\times X\isoto \xppp,\,(\sigma,\tau,x)\mapsto(x,\pi(x),\tau\pi(x),(\sigma\tau)\pi(x)),
\end{equation}
where we have used the commutativity of \eqref{gal3} and \eqref{comm} to obtain the indicated set-theoretic formulas.

Now let $\rho\colon \g\to\aut\!\left(X/S\right)$ be an action of $\g$ on  $X$ over $S$ which is compatible with the canonical action of $\g$ on $\sp$ over $S$, i.e., if $\g\times X\to X$ is the $S$-morphism induced by $\rho$, then the following diagram of $S$-morphisms commutes 
\begin{equation}\label{act2}
\xymatrix@1@R=30pt@C=40pt{
\g\times X\ar[r]\ar[d]_{1_{\lbe\g}\times\e \pi}& X \ar[d]^{\pi} \\
\g\times\sp\ar[r]^{a} &\sp\,,
}
\end{equation}
where $a$ is given by \eqref{a}. We will show that $\rho$ defines a descent datum on $X$ relative to $f$ by constructing a diagram of the form \eqref{mcart} with cartesian squares such that \eqref{q1}, \eqref{q2} (where $\varphi$ is the covering datum on $X$ determined by the right-hand square in \eqref{mcart} for $i=2$), \eqref{q12}, \eqref{q13}, \eqref{eq2} and \eqref{eq3} hold (see the previous Section).

\smallskip

We begin by noting that \eqref{act2} may be written as
\begin{equation}\label{act2.1}
\xymatrix@1@R=30pt@C=40pt{
\displaystyle\coprod_{\sigma\in\g} X\ar[r]^(.55){\coprod \rho(\sigma)} \ar[d]_(.55){\coprod \pi}& X \ar[d]^{\pi}\\
\displaystyle\coprod_{\sigma\in\g} \sp\ar[r]^{\coprod \sigma} & \sp.
}
\end{equation}
Now since 
\begin{equation}\label{srho}
\xymatrix@1@R=35pt@C=45pt{
X\ar[r]^{\rho(\sigma)}_{\sim} \ar[d]_{\pi}& X \ar[d]^{\pi} \\
\sp\ar[r]^{\sigma}_{\sim} &\sp
}
\end{equation}
is cartesian for every $\sigma\in\g$, Lemma \eqref{ucart} shows that the equivalent diagrams \eqref{act2} and \eqref{act2.1} are cartesian as well. We now observe that, if
\begin{equation}\label{q2a}
q_{\le 2}=\big(\e\coprod\rho(\sigma)\big)\be\circ \vartheta_{\be X^{\prime\prime}}^{-1},
\end{equation}
then the cartesian square \eqref{act2} decomposes as
\begin{equation}\label{twin}
\xymatrix@C=35pt{
\g\times X\ar[r]^(.55){\vartheta_{\be X^{\prime\prime}}}_{\sim}\ar@/^1.7pc/[rr]^{\coprod\rho(\sigma)} \ar[d]_{1_{\lbe\g}\times\e \pi\,=\,\vartheta^{*}\be(\e p_{\le 1}^{*}(\pi))}& \xpp\ar[r]^(0.5){q_{\le 2}}\ar[d]_{p_{1}^{*}\lbe(\pi)}& X\ar[d]^{\pi}\\
\g\times \sp\ar[r]^(.5){\vartheta}_{\sim}\ar@/_1.2pc/[rr]_{a}& \spp\ar[r]^(0.5){p_{\le 2}} & \sp,}
\end{equation}
where the lower part of the diagram commutes by the commutativity of \eqref{gal2}. We conclude that the right-hand square in \eqref{twin} is cartesian. Thus, setting $q_{\le 1}=p_{\le 1,\lle X}$ (whence \eqref{q1} holds), there exist cartesian diagrams for  $i=1$ and $2$
\[
\xymatrix@1@R=30pt@C=40pt{
\xpp\ar[r]^(.5){q_{\le i}} \ar[d]_{p_{1}^{*}\lbe(\pi)}& X \ar[d]^{\pi}\\
\spp\ar[r]^(.5){p_{\le i}} &\sp\,,
}
\]
that define a covering datum $\varphi\colon \xpp\isoto p_{\le 2}^{*}X$ on $X$ relative to $f\e$ such that \eqref{q2} holds.

Now let $\widetilde{q}_{jk}\colon \g\times\g\times X\to \g\times X$ be given by the formulas
\begin{eqnarray}\label{pt1}
\widetilde{q}_{12}(\sigma,\tau,x)&=&(\tau,x\e),\\
\label{pt2}
\widetilde{q}_{13}(\sigma,\tau,x)&=&(\sigma\tau,x\e),\\
\label{pt3}
\widetilde{q}_{23}(\sigma,\tau,x)&=&(\sigma,\rho(\tau)x\e).
\end{eqnarray}
Then \eqref{q1}, \eqref{galp}, \eqref{pt1} and \eqref{pt3} yield
\begin{equation}\label{once}
q_{\le 1}\circ \vartheta_{\be X^{\prime\prime}}\circ \widetilde{q}_{\le 23} =\big(\e\coprod\rho(\sigma)\big)\be\circ \widetilde{q}_{\le 12}.
\end{equation}
Further, since $\rho(\sigma)(\e\rho(\tau)x)=\rho(\e\sigma\tau)x$ for all $(\sigma,\tau,x)\in\g\times\g\times X$, we have
\begin{equation}\label{rnice}
\big(\e\coprod\rho(\sigma)\big)\be\circ \widetilde{q}_{\le 23} =\big(\e\coprod\rho(\sigma)\big)\be\circ \widetilde{q}_{\le 13}.
\end{equation}
Define
\[
q_{jk}=\vartheta_{\be X^{\prime\prime}}\circ\widetilde{q}_{\le jk}\circ \varrho_{\lbe X^{\prime\prime\prime}}^{-1}.
\]
Then \eqref{eq2} and \eqref{eq3} follow at once from \eqref{q2a}. \eqref{once} and \eqref{rnice}. Further, since $\widetilde{q}_{\le jk}=\widetilde{p}_{\le jk,\e \g\times X}$ for $(\e j,k)=(1,2)$ and $(1,3)$, the commutativity of \eqref{comm} shows that $q_{\le jk}= p_{\le jk,\le X^{\prime\prime}}$ for such $(\e j,k)$, i.e., \eqref{q12} and \eqref{q13} hold. Next, the diagram 
\begin{equation}\label{qt}
\xymatrix@1@R=30pt@C=40pt{
\g\times\g\times X\ar[r]^(.6){\widetilde{q}_{\le jk}} \ar[d]_{1_{\g}\times 1_{\g}\times \pi} & \g \times X \ar[d]^{1_{\g}\times \pi}\\
\g\times \sp\ar[r]^(.6){\widetilde{p}_{jk}} &\sp
}
\end{equation}
is cartesian for $(\e j,k)=(1,2),(1,3)$ and $(2,3)$. This is clear if $(\e j,k)=(1,2)$ or $(1,3)$. If $(\e j,k)=(2,3)$, then \eqref{qt} is cartesian because \eqref{act2} is cartesian.
Now \eqref{qt} decomposes as 
\begin{equation}\label{twin2}
\xymatrix@C=35pt{
\g\times\g\times X\ar[r]^(.6){\varrho_{\lbe X^{\prime\prime\prime}}}_(.55){\sim}\ar@/^2.5pc/[rrr]^{\widetilde{q}_{\le jk}} \ar[dd]_{1_{\g}\times 1_{\g}\times \pi\,=\,\varrho^{*}\be(\e p_{\le 12}^{*}\e p_{\le 1}^{*}(\pi))}& \xppp\ar[r]^(0.5){q_{\le jk}}\ar[dd]_{p_{12}^{*}\e p_{1}^{*}\lbe(\pi)}& \xpp\ar[r]^(0.4){\vartheta_{\be X^{\prime\prime}}^{-1}}_(.4){\sim}\ar[dd]_{p_{1}^{*}\lbe(\pi)}&\g\times X\ar[dd]^{1_{\g}\times \pi}\\
&&\\
\g\times\g\times \sp\ar[r]^(.6){\varrho}_(.6){\sim}\ar@/_1.5pc/[rrr]_{\widetilde{p}_{jk}}& \sppp\ar[r]^(0.5){p_{\le jk}} & \spp\ar[r]^(.4){\vartheta^{-1}}_(.4){\sim}&\g\times\sp\,,}
\end{equation}
where the bottom part of the diagram commutes by the commutativity of \eqref{comm}. Consequently, the central square above is cartesian. Thus we obtain the desired commutative diagrams with cartesian squares
\[
\xymatrix@1@R=30pt@C=40pt{
\xppp\ar[r]^(.5){q_{\le jk}} \ar[d]_{p_{12}^{*}\e p_{1}^{*}\lbe(\pi)}& \xpp\ar[r]^(.5){q_{\le i}}\ar[d]_{p_{1}^{*}\lbe(\pi)}& X\ar[d]^{\pi}\\
\sppp\ar[r]^(.5){p_{jk}} &\spp\ar[r]^(.5){p_{\le i}}&\sp.\\
}
\]
such that \eqref{q1}, \eqref{q2}, \eqref{q12}, \eqref{q13}, \eqref{eq2} and \eqref{eq3} hold.

\smallskip

The descent datum $\varphi\colon \xpp\isoto p_{2}^{*}X$ on $X$ relative to $f$ thus associated to $\rho$ may be described (set-theoretically) as follows. By \eqref{q2} and \eqref{q2a}, we have
\begin{equation}\label{for}
\coprod\rho(\sigma)=p_{\le 2,\le X}\circ\varphi\circ \be \vartheta_{\be X^{\prime\prime}}.
\end{equation}
It then follows that $\varphi$ is given by the formula
\[
\varphi(x,\pi(x),s^{\e\prime})=(\e\rho(\sigma)x,\pi(x),s^{\e\prime}\le),
\]
where $\sigma$ is the unique element of $\g$ such that $s^{\e\prime}=\sigma\pi(x)$.

\begin{example}\label{tact} Let $Y$ be an $S$-scheme. Then $\yp=Y\times_{S}\le\sp$ is canonically endowed with an action of $\g$ over $S$ that is compatible with $a$, namely $\coprod(1_{Y}\!\times_{S}\sigma)\colon\g\times\yp\to\yp$. The associated descent datum on $\yp$ (relative to $f\e$) is the isomorphism of $\spp$-schemes  $c_{\e Y}\colon
p_{\le 1}^{*}\yp\isoto p_{\le 2}^{*}\yp$ \eqref{cy}.	
\end{example}

\section{Galois descent of morphisms}\label{4}

We keep the notation and hypotheses of the previous Section. In this Section we generalize the standard result \cite[Proposition 2.8]{j} on the Galois descent of morphisms of $k$-schemes, where $k$ is a field, to an arbitrary base scheme $S$.

\smallskip

For $i=1$ or 2, let $\pi_{i}\colon X_{i}\to\sp$ be an $\sp$-scheme equipped 
with an action $\rho_{\e i}\colon \g\to\aut\!\left(X_{i}/S\right)$ that is compatible with the canonical action of $\g$ on $\sp$ over $S$. If $\delta\colon X_{1}\to X_{2}$ is an $\sp$-morphism, i.e., $\pi_{2}\circ\delta=\pi_{1}$, then the commutativity of \eqref{srho} (for both $\rho_{1}$ and $\rho_{\le 2}$) shows that $\rho_{2}(\sigma)\circ \delta\circ \rho_{1}(\sigma)^{-1}\colon X_{1}\to X_{2}$ is a morphism of $\sp$-schemes for every $\sigma\in\g$. Thus we may define a left action of $\g$ on the set $\Hom_{S^{\lle\prime}}(X_{1},X_{2})$ by
\begin{equation}\label{act3}
\g\times \Hom_{S^{\lle\prime}}(X_{1},X_{2})\to \Hom_{S^{\lle\prime}}(X_{1},X_{2}), (\sigma,\delta)\mapsto \rho_{2}(\sigma)\circ \delta\circ \rho_{1}(\sigma)^{-1}.
\end{equation}

Now, for $i=1$ and $2$, let $\varphi_{\le i}\colon p_{\le 1}^{*}X_{i}\isoto p_{\le 2}^{*}X_{i}$ be the descent datum on $X_{i}$ associated to $\rho_{\e i}$ in the previous Section. Note that $p_{\le 2}^{*}(\pi_{\e i})\circ\varphi_{\e i}=p_{\le 1}^{*}(\pi_{i})$ for $i=1$ and 2.

\begin{proposition} Let $\delta\in \Hom_{S^{\lle\prime}}(X_{1},X_{2})$. Then $\delta$ is invariant under the action of $\g$ if, and only if, the diagram
\begin{equation}\label{poq}
\xymatrix@1@R=30pt@C=40pt{
p_{\le 1}^{*}X_{1}\e\ar[r]^{\varphi_{\le 1}}_{\sim} \ar[d]_(.55){p_{\le 1}^{*}(\delta)}& p_{\le 2}^{*}X_{1} \ar[d]^{p_{\le 2}^{*}(\delta)}\\
p_{\le 1}^{*}X_{2}\,\ar[r]^{\varphi_{\le 2}}_{\sim} & p_{\le 2}^{*}X_{2}
}
\end{equation}
commutes.
\end{proposition}
\begin{proof} By the definition \eqref{act3}, we need to show that \eqref{poq} commutes if, and only if,
\begin{equation}\label{geq}
\xymatrix@1@R=30pt@C=40pt{
\g\times X_{1}\e\ar[r]^(.58){\coprod \rho_{\le 1}\lbe(\sigma)} \ar[d]_(.55){1_{\g}\times \delta}& X_{1} \ar[d]^{\delta}\\
\g\times X_{2}\,\ar[r]^(.58){\coprod \rho_{\le 2}(\sigma)} & X_{2}
}
\end{equation}
commutes. By \eqref{for} applied to both $\rho_{\e 1}$ and $\rho_{\e 2}$, the preceding diagram decomposes as
\begin{equation}\label{geq2}
\xymatrix@C=35pt{
\g\times X_{1}\ar[r]^(.5){\vartheta_{\be p_{\lbe 1}^{\lbe *}\be X_{1}}}_{\sim} \ar[d]_{1_{\g}\times\delta\,=\,\vartheta_{\!\be p_{\lbe 1}^{\lbe *}\be X_{2}}^{\le *}\lbe(\e p_{\lle 1}^{*}(\delta))}& p_{\lle 1}^{*}X_{1}\ar[r]^(.5){\varphi_{\le 1}}_{\!\sim}\ar[d]_{p_{\lle 1}^{*}(\delta)}& p_{\lle 2}^{*}X_{1}\ar[r]^{p_{\le 2,X_{1}}}\ar[d]^{p_{\lle 2}^{*}(\delta)}& X_{1}\ar[d]^{\delta}\\
\g\times X_{2}\ar[r]^(.5){\vartheta_{\be p_{\lbe 1}^{\lbe *}\be X_{2}}}_{\sim}& p_{\le 1}^{*}X_{2}\ar[r]^(0.5){\varphi_{2}}_{\sim} & p_{\le 2}^{*}X_{2}\ar[r]^(.5){p_{\lle 2, X_{2}}}& X_{2}\,,}
\end{equation}
where the left-hand and right-hand squares commute. Thus, if \eqref{poq} commutes, then \eqref{geq} commutes as well. Conversely, assume that \eqref{geq}, i.e., the outer diagram in \eqref{geq2}, commutes. To show that \eqref{poq} commutes, it suffices to check that the diagram with cartesian square
\begin{equation}\label{kart}
\xymatrix{p_{\lle 1}^{*}X_{1}\ar[dr]_(.45){h}\ar@/^1.6pc/[drrr]^(.6){\hskip .5cm p_{\le 2,\le X_{2}}\e\circ\e \varphi_{\e 2}\e\circ\e p_{\le 1}^{*}(\delta)}\ar@/_1.8pc/[ddr]_{p_{\le 2}^{*}(\pi_{2})\e\circ\e\varphi_{2}\e\circ\e p_{1}^{*}\lbe(\delta)}\\
&p_{\lle 2}^{*}X_{2}\ar[d]_{p_{\lle 2}^{*}(\pi_{2})}\ar[rr]^(.55){p_{\le 2,X_{2}}}&& X_{2}\ar[d]^{\pi_{\e 2}}\\
&\spp\ar[rr]_(.48){p_{\le 2}}&&\sp\,,
}
\end{equation}
commutes for $h=\varphi_{\e 2}\e\circ\e p_{\le 1}^{*}(\delta)$ and $h= p_{\le 2}^{*}(\delta)\circ \varphi_{\e 1}$. The above diagram clearly commutes if $h=\varphi_{\e 2}\e\circ\e p_{\le 1}^{*}(\delta)$. Now, since $\vartheta_{\be p_{\lbe 1}^{\lbe *}\be X_{1}}$ is an isomorphim and the outer diagram and left-hand square in \eqref{geq2} commute, we have
$p_{\le 2,\e X_{2}}\circ p_{\le 2}^{*}(\delta)\circ\varphi_{\e 1}=p_{\le 2,\e X_{2}}\circ \varphi_{\e 2}\circ p_{\le 1}^{*}(\delta)$, i.e., the top triangle of diagram \eqref{kart} commutes when $h= p_{\le 2}^{*}(\delta)\circ \varphi_{\e 1}$. The commutativity of the lower triangle in 
\eqref{kart} when $h= p_{\le 2}^{*}(\delta)\circ \varphi_{\e 1}$ can be checked using the identities $\pi_{2}\circ\delta=\pi_{1}$ and $p_{\le 2}^{*}(\pi_{\e i})\circ\varphi_{\e i}=p_{\le 1}^{*}(\pi_{i})$ ($i=1$ and 2).
\end{proof}

\smallskip

Recall now that the descent datum $\varphi_{\le i}\colon p_{\le 1}^{*}X_{i}\isoto p_{\le 2}^{*}X_{i}$ is said to be {\it effective} if there exist $S$-schemes $Y_{i}$ and $\sp$-isomorphisms $\theta_{i}\colon X_{i}\isoto Y_{i}^{\le\prime}$ such that the diagram
\[
\xymatrix@1@R=30pt@C=40pt{
p_{\lle 1}^{*}X_{i}\ar[r]^(.5){\varphi_{i}}_{\!\sim} \ar[d]_{p_{\lle 1}^{*}\lbe(\theta_{i})}^{\sim}& p_{\lle 2}^{*}X_{i} \ar[d]^{p_{\lle 2}^{*}\lbe(\theta_{i})}_{\sim}\\
p_{\le 1}^{*}Y_{i}^{\le\prime}\ar[r]^(.5){c_{\le Y_{\lbe i}}}_{\!\sim} & p_{\le 2}^{*}Y_{i}^{\le\prime}.
}
\]
commutes. If this is the case, then we say that {\it $X_{i}$ descends to $Y_{i}$} (or {\it to $S\e$}). By \cite[VIII, Corollary 7.6]{sga1}, $X_{i}$ descends to $S$ if $\pi_{i}\colon X_{i}\to\sp$ is quasi-projective.

\begin{corollary} Assume that, for $i=1$ and $2$, $X_{i}$ descends to $Y_{i}$ and let $\theta_{i}\colon X_{i}\isoto Y_{i}^{\le\prime}$ be the corresponding isomorphism of $\sp$-schemes. Let $\delta\colon X_{1}\to X_{2}$ be an $\sp$-morphism and define $\varepsilon\colon Y_{1}^{\le\prime}\to Y_{2}^{\le\prime}$ by the commutativity of the diagram
\[
\xymatrix@1@R=30pt@C=40pt{
X_{1}\ar[r]^(.5){\delta} \ar[d]_{\theta_{1}}^{\sim}& X_{2} \ar[d]^{\theta_{2}}_{\sim}\\
Y_{1}^{\le\prime}\ar[r]^(.5){\varepsilon} &Y_{2}^{\le\prime}.
}
\]
Then $\varepsilon=\psi\times_{S}\sp$ for some $S$-morphism $\psi\colon Y_{1}\to Y_{2}$, if, and only if, $\varepsilon$ is invariant under $\g$ \eqref{act3}, i.e., for every $\sigma\in\g$, the diagram
\[
\xymatrix@1@R=30pt@C=40pt{
Y_{1}^{\le\prime}\e\ar[r]^(.5){1_{Y_{\lbe 1}}\times\e\sigma} \ar[d]_{\varepsilon}& Y_{1}^{\le\prime} \ar[d]^{\varepsilon}\\
Y_{2}^{\le\prime}\,\ar[r]^(.5){1_{Y_{\lbe 2}}\times\e\sigma} & Y_{2}^{\le\prime}
} 
\]
commutes.
\end{corollary}
\begin{proof} By \cite[Theorem 5.2 and comment after the statement]{sga1}, $\varepsilon=\psi\times_{S}\sp$ for some $S$-morphism $\psi\colon Y_{1}\to Y_{2}$, if, and only if, the diagram (which is an instance of \eqref{poq})
\begin{equation}\label{ppf2}
\xymatrix@1@R=30pt@C=40pt{
p_{\le 1}^{*}Y_{1}^{\le\prime}\e\ar[r]^{c_{\e Y_{\lbe 1}}}_{\sim} \ar[d]_(.55){p_{\le 1}^{*}(\varepsilon)}& p_{\le 2}^{*}Y_{1}^{\le\prime} \ar[d]^{p_{\le 2}^{*}(\varepsilon)}\\
p_{\le 1}^{*}Y_{2}^{\le\prime}\,\ar[r]^{c_{\e Y_{\lbe 2}}}_{\sim} & p_{\le 2}^{*}Y_{2}^{\le\prime}
}
\end{equation}
commutes (see the next remark). By the proposition, the latter is the case if, and only if, $\varepsilon$ is invariant under the action of $\g$. 	
\end{proof}

\begin{remark} In \cite[Theorem 5.2 and comment after the statement]{sga1}, the schemes $p_{\lle 1}^{*}Y_{i}^{\le\prime}$ and $p_{\lle 2}^{*}Y_{i}^{\le\prime}$ have been identified via $c_{\e Y_{\lbe i}}$. Thus the condition in [loc.cit.] that $p_{\le 1}^{*}(\varepsilon)$ and $p_{\le 2}^{*}(\varepsilon)$ be  equal is indeed equivalent to the commutativity of diagram \eqref{ppf2}.
\end{remark}


\begin{thebibliography}{3}
\bibitem[BLR]{blr} Bosch, S., L\"utkebohmert, W. and Raynaud, M.: N\'eron models. Erg. der Math. Grenz. {\bf{21}}, Springer-Verlag,  Berlin, 1990.


\bibitem[GW]{gw} G\"ortz, U. and Wedhorn, T.: Algebraic geometry I. Schemes with examples and exercises. Advanced Lectures in Mathematics, Vieweg + Teubner, Wiesbaden, 2010.

\bibitem[SGA1]{sga1} Grothendieck, A.: Rev\^etements \'etales et groupe fondamental (SGA 1). S\'eminaire de g\'eom\'etrie alg\'ebrique du Bois Marie 1960--61. Lecture Notes in Math. {\bf{224}},Springer-Verlag 1971.



\bibitem[$\text{EGA I}_{\le\text{new}}$\e]{ega1} Grothendieck, A. and
Dieudonn\'e, J.: \'El\'ements de g\'eom\'etrie alg\'ebrique I. Le langage des
sch\'emas. Grundlehren Math. Wiss. {\bf{166}}, Springer-Verlag, Berlin, 1971.


\bibitem[J]{j} Jahnel, J.: The Brauer-Severi variety associated with a central simple algebra (unpublished). Available at \url{https://www.math.uni-bielefeld.de/lag/man/052.pdf}



\end{thebibliography}
\end{document}